\documentclass{amsart}

\usepackage[dvips]{graphicx}

\usepackage{charter}
\usepackage[a4paper,top=46mm,bottom=46mm,left=37mm,right=37mm]{geometry}

\usepackage{latexsym,amsfonts} 
\usepackage{amsmath,amssymb,graphics,setspace}
\usepackage{amsthm}
\usepackage{mathrsfs} 

\usepackage[ruled,vlined,linesnumbered]{algorithm2e}

\usepackage{marvosym}

\usepackage{pstricks,pst-text,pst-grad,pst-node,pst-3dplot,pstricks-add,pst-poly,pst-coil,pst-bar,pst-all,pst-plot,pst-2dplot} 
\usepackage{pst-fun,pst-blur}

\usepackage{picins,graphicx}
\usepackage{paralist}
\usepackage{color}

\usepackage[verbose]{wrapfig}

\usepackage{subfigure}
\renewcommand{\thesubfigure}{\thefigure.\arabic{subfigure}}
\makeatletter
\renewcommand{\p@subfigure}{}
\renewcommand{\@thesubfigure}{\thesubfigure:\hskip\subfiglabelskip}
\makeatother

\usepackage{hyperref}
\hypersetup{linktocpage=true,colorlinks=true,linkcolor=blue,citecolor=blue,pdfstartview={XYZ 1000 1000 1}}

\usepackage{floatflt,graphicx}

\usepackage{graphicx}
\makeatletter
\DeclareFontFamily{U}{tipa}{}
\DeclareFontShape{U}{tipa}{bx}{n}{<->tipabx10}{}
\newcommand{\arc@char}{{\usefont{U}{tipa}{bx}{n}\symbol{62}}}%

\newcommand{\arc}[1]{\mathpalette\arc@arc{#1}}

\newcommand{\arc@arc}[2]{%
  \sbox0{$\m@th#1#2$}%
  \vbox{
    \hbox{\resizebox{\wd0}{\height}{\arc@char}}
    \nointerlineskip
    \box0
  }%
}
\makeatother

\makeatletter
\newcommand{\doublewedge}{\big@doubleop{\wedge}}
\newcommand{\big@doubleop}[1]{%
  \DOTSB\mathop{\mathpalette\big@doubleop@aux{#1}}\slimits@
}

\newcommand\big@doubleop@aux[2]{%
  \sbox\z@{$\m@th#1#2$}%
  \makebox[1.35\wd\z@][s]{$\m@th#1#2\hss#2$}%
}

\newcommand{\norm}[1]{\left\|#1\right\|}  

\newcommand{\rb}{\mbox{rb}}
\newcommand{\rbx}{\mbox{rbx}}
\newcommand{\rbNrv}{\mbox{rbNrv}}
\newcommand{\dxnear}{\delta_{\norm{\Phi}}} 
\newcommand{\xnear}{\norm{\Phi}} 
\newcommand{\cl}{\mbox{cl}}

\newcommand{\Int}{\mbox{int}}
\newcommand{\bdy}{\mbox{bdy}}

\newcommand{\Nrv}{\mbox{Nrv}}
\newcommand{\dnear}{\delta_{\Phi}} 
\newcommand{\dcap}{\mathop{\cap}\limits_{\Phi}} 
\newcommand{\dxcap}{\mathop{\cap}\limits_{\norm{\Phi}}} 

\newcommand{\sh}{\mbox{sh}}
\newcommand{\cx}{\mbox{cx}}
\newcommand{\cyc}{\mbox{cyc}}

\newcommand{\vNrv}{\mbox{vNrv}}

\newtheorem{example}{Example}
\newtheorem{remark}{Remark}
\newtheorem{definition}{Definition}
\newtheorem{lemma}{Lemma}
\newtheorem{theorem}{Theorem}

\newtheorem{corollary}{Corollary}

\usepackage{xcolor}
\definecolor{light}{gray}{0.80}

\begin{document}

\title[Ribbon Complexes and Approximate Proximities]{Ribbon Complexes\ \&\ their Approximate Descriptive Proximities.\\   Ribbon\ \&\ Vortex Nerves, Betti Numbers and Planar Divisions}

\author[James F. Peters]{James F. Peters}
\address{
Computational Intelligence Laboratory,
University of Manitoba, WPG, MB, R3T 5V6, Canada and
Department of Mathematics, Faculty of Arts and Sciences, Ad\.{i}yaman University, 02040 Ad\.{i}yaman, Turkey}
\thanks{The research has been supported by the Natural Sciences \&
Engineering Research Council of Canada (NSERC) discovery grant 185986 
and Instituto Nazionale di Alta Matematica (INdAM) Francesco Severi, Gruppo Nazionale per le Strutture Algebriche, Geometriche e Loro Applicazioni grant 9 920160 000362, n.prot U 2016/000036 and Scientific and Technological Research Council of Turkey (T\"{U}B\.{I}TAK) Scientific Human
Resources Development (BIDEB) under grant no: 2221-1059B211301223.}


\date{}

\dedicatory{Dedicated to Enrico Betti and Som Naimpally}

\begin{abstract}
This article introduces planar ribbons, Vergili ribbon complexes and ribbon nerves in Alexandroff-Hopf-Whitehead CW (Closure finite Weak) topological spaces.    A {\em planar ribbon} (briefly, {ribbon}) in a CW space is the closure of a pair of nesting, non-concentric filled cycles that includes boundary but does not include the interior of the inner cycle.  Each planar ribbon has its own distinctive shape determined by its outer and inner boundaries and the interior within its boundaries. A {\em Vergili ribbon complex} (briefly, {\em ribbon complex}) in a CW space is a non-void collection of countable planar ribbons.  A {\em ribbon nerve} is a nonvoid collection of planar ribbons (members of a ribbon complex) that have nonempty intersection.  A planar CW space is a  non-void collection of cells (vertexes, edges and filled triangles) that may or may not be attached to each other and which satisfy Alexandroff-Hopf-Whitehead containment and intersection conditions. In the context of CW spaces, planar ribbons, ribbon complexes, ribbon nerves and vortex nerves are characterized by Betti numbers derived from standard Betti numbers $\mathcal{B}_0$ (cell count), $\mathcal{B}_1$ (cycle count) and $\mathcal{B}_2$ (hole count), namely, $\mathcal{B}_{rb}$ and $\mathcal{B}_{rbNrv}$ introduced in this paper.  Results are given for collections of ribbons and ribbon nerves in planar CW spaces equipped with an approximate descriptive proximity, division of the plane into three bounded regions by a ribbon and Brouwer fixed points on ribbons.  In addition, the homotopy type of ribbon nerves are introduced.   
\end{abstract}

\maketitle
\tableofcontents

\section{Introduction}
This paper introduces planar ribbons and ribbon complexes in a CW topological space $K$. A \emph{cell complex} is a nonempty collection of cells.   A \emph{cell} in the Euclidean plane is either a 0-cell (vertex) or 1-cell (edge) or 2-cell (filled triangle).   A nonvoid collection of cell complexes $K$ has a \emph{Closure finite Weak} (CW) topology, provided $K$ is Hausdorff (every pair of distinct cells is contained in disjoint neighbourhoods~\cite[\S 5.1, p. 94]{Naimpally2013}) and the collection of cell complexes in $K$ satisfy the Alexandroff-Hopf-Whitehead~\cite[\S III, starting on page 124]{AlexandroffHopf1935Topologie},~\cite[pp. 315-317]{Whitehead1939homotopy}, ~\cite[\S 5, p. 223]{Whitehead1949BAMS-CWtopology} conditions, namely, containment (the closure of each cell complex is in $K$) and intersection (the nonempty intersection of cell complexes is in $K$).

\begin{figure}[!ht]
\centering
\subfigure[Planar ribbon $\rb E$, {\textcolor{gray!80}{\Large $\boldsymbol{\bullet}$}} = hole]
 {\label{fig:rbE}
\begin{pspicture}
(-1.5,-0.5)(4.0,4.0)
\psframe[linewidth=0.75pt,linearc=0.25,cornersize=absolute,linecolor=blue](-1.25,-0.25)(3.25,4)
\psline*[linestyle=solid,linecolor=green!30]%
(0,0)(1,0.5)(2.0,0.0)(3.0,0.5)(3.0,1.5)(2.0,2.0)(1,1.5)(0,2)
(-1,1.5)(-1,0.5)(0,0)
\psline[linestyle=solid,linecolor=black]%
(0,0)(1,0.5)(2.0,0.0)(3.0,0.5)(3.0,1.5)(2.0,2.0)(1,1.5)(0,2)
(-1,1.5)(-1,0.5)(0,0)
\psdots[dotstyle=o,dotsize=2.2pt,linewidth=1.2pt,linecolor=black,fillcolor=gray!80]%
(0,0)(1,0.5)(2.0,0.0)(3.0,0.5)(3.0,1.5)(2.0,2.0)(1,1.5)(0,2)
(-1,1.5)(-1,0.5)(0,0)
\psline[linestyle=solid](1,1.5)(2,2)\psline[arrows=<->](-1,1.7)(-0.3,2.0)\psline[arrows=<->](0.2,2.05)(0.8,1.75)
\psline*[linestyle=solid,linecolor=white]%
(0,0.25)(1,0.75)(2.0,0.25)(2.5,0.5)(2.5,0.75)(2.0,1.35)(1,1.25)(0,1.5)
(-.55,1.25)(-.55,0.75)(0,0.25)
\psline[linestyle=solid,linecolor=black]%
(0,0.25)(1,0.75)(2.0,0.25)(2.5,0.5)(2.5,0.75)(2.0,1.35)(1,1.25)(0,1.5)
(-.55,1.25)(-.55,0.75)(0,0.25)
\psdots[dotstyle=o,dotsize=2.2pt,linewidth=1.2pt,linecolor=black,fillcolor=gray!80]%
(0,0.25)(1,0.75)(2.0,0.25)(2.5,0.5)(2.5,0.75)(2.0,1.35)(1,1.25)(0,1.5)
(-.55,1.25)(-.55,0.75)(0,0.25)
\psdots[dotstyle=o,dotsize=5.2pt,linewidth=1.2pt,linecolor=black,fillcolor=gray!90]%
(-.80,1.05)(2.80,0.55)
\rput(-1.0,3.75){\footnotesize $\boldsymbol{K}$}
\rput(0.0,1.75){\footnotesize $\boldsymbol{rb E}$}
\rput(2.8,1.85){\footnotesize $\boldsymbol{cyc A}$}
\rput(2.5,1.25){\footnotesize $\boldsymbol{cyc B}$}
\end{pspicture}}\hfil
\subfigure[Ribbon nerve rbNrv K, {\textcolor{gray!80}{\Large $\boldsymbol{\bullet}$}} = hole]
 {\label{fig:rbNrvE}
\begin{pspicture}
(-1.5,-0.5)(4.0,4.0)
\psframe[linewidth=0.75pt,linearc=0.25,cornersize=absolute,linecolor=blue](-1.25,-0.25)(3.75,4)
\psline*[linestyle=solid,linecolor=green!30]%
(0,0)(1,0.5)(2.0,0.0)(3.0,0.5)(3.0,1.5)(2.0,2.0)(1,1.5)(0,2)
(-1,1.5)(-1,0.5)(0,0)
\psline[linestyle=solid,linecolor=black]%
(0,0)(1,0.5)(2.0,0.0)(3.0,0.5)(3.0,1.5)(2.0,2.0)(1,1.5)(0,2)
(-1,1.5)(-1,0.5)(0,0)
\psdots[dotstyle=o,dotsize=2.2pt,linewidth=1.2pt,linecolor=black,fillcolor=gray!80]%
(0,0)(1,0.5)(2.0,0.0)(3.0,0.5)(3.0,1.5)(2.0,2.0)(1,1.5)(0,2)
(-1,1.5)(-1,0.5)(0,0)
\psline[linestyle=solid](1,1.5)(2,2)\psline[arrows=<->](-1,1.7)(-0.3,2.0)\psline[arrows=<->](0.2,2.05)(0.8,1.75)
\psline*[linestyle=solid,linecolor=white]%
(0,0.25)(1,0.75)(2.0,0.25)(2.5,0.5)(2.5,1.25)(2.0,1.55)(1,1.25)(0,1.5)
(-.55,1.25)(-.55,0.75)(0,0.25)
\psline[linestyle=solid,linecolor=black]%
(0,0.25)(1,0.75)(2.0,0.25)(2.5,0.5)(2.5,1.25)(2.0,1.55)(1,1.25)(0,1.5)
(-.55,1.25)(-.55,0.75)(0,0.25)
\psdots[dotstyle=o,dotsize=2.2pt,linewidth=1.2pt,linecolor=black,fillcolor=gray!80]%
(0,0.25)(1,0.75)(2.0,0.25)(2.5,0.5)(2.5,1.25)(2.0,1.55)(1,1.25)(0,1.5)
(-.55,1.25)(-.55,0.75)(0,0.25)
\psline*[linestyle=solid,linecolor=orange!35]%
(2,2)(1,2.25)(1.0,3.55)(2.25,3.85)(3.5,3.25)(3.5,2.25)(2,2)
\psline[linestyle=solid,linecolor=black]%
(2,2)(1,2.25)(1.0,3.55)(2.25,3.85)(3.5,3.25)(3.5,2.25)(2,2)
\psdots[dotstyle=o,dotsize=2.2pt,linewidth=1.2pt,linecolor=black,fillcolor=gray!80]%
(2,2)(1,2.25)(1.0,3.55)(2.25,3.85)(3.5,3.25)(3.5,2.25)(2,2)
\psdots[dotstyle=o,dotsize=2.5pt,linewidth=1.2pt,linecolor=black,fillcolor=red!80]%
(2,2)
\psline*[linestyle=solid,linecolor=white]%
(2,2.25)(1.25,2.5)(1.25,3.0)(2.25,3.25)(3.25,3.0)(3.25,2.35)(2,2.25)
\psline[linestyle=solid,linecolor=black]%
(2,2.25)(1.25,2.5)(1.25,3.0)(2.25,3.25)(3.25,3.0)(3.25,2.35)(2,2.25)
\psdots[dotstyle=o,dotsize=2.2pt,linewidth=1.2pt,linecolor=black,fillcolor=gray!80]%
(2,2.25)(1.25,2.5)(1.25,3.0)(2.25,3.25)(3.25,3.0)(3.25,2.35)(2,2.25)
\psdots[dotstyle=o,dotsize=5.2pt,linewidth=1.2pt,linecolor=black,fillcolor=gray!90]%
(-.80,1.05)(2.80,0.85)(2.30,3.41)(2.50,3.61)(2.80,3.31)
\rput(-1.0,3.75){\footnotesize $\boldsymbol{K'}$}
\rput(0.0,1.75){\footnotesize $\boldsymbol{rb A}$}
\rput(2.0,1.85){\footnotesize $\boldsymbol{a}$}
\rput(1.5,3.35){\footnotesize $\boldsymbol{rb B}$}
\end{pspicture}}
\caption[]{Sample planar ribbon structure and ribbon nerve structure}
\label{fig:vortexCycles}
\end{figure}

\begin{definition} Planar Ribbon.\\
Let $\cyc A, \cyc B$ be nesting filled cycles (with $\cyc B$ in the interior of $\cyc A$) defined on a collection of vertices $E$ on a finite, bounded, planar region in a CW space $K$. A \emph{planar ribbon} $E$ (denoted by $\rb E$) is defined by
\[
\rb E = \overbrace{
\left\{\cl(\cyc A)\setminus \left\{\cl(\cyc B)\setminus \Int(\cyc B)\right\}: \bdy(\cyc B)\subset \cl(\rb E)\right\}.}^{\mbox{\textcolor{blue}{\bf $\bdy(\cl(\cyc B))$ defines the inner boundary of $\cl(\rb E)$.}}}\mbox{\textcolor{blue}{\Squaresteel}}
\]
\end{definition}


\begin{figure}[!ht]
\centering
\begin{pspicture}
(-1.5,-0.5)(8.0,4.0)
\psframe[linewidth=0.75pt,linearc=0.25,cornersize=absolute,linecolor=blue](-1.25,-0.25)(7.25,4)
\psline*[linestyle=solid,linecolor=green!30]%
(0,0)(1,0.5)(2.0,0.0)(3.0,0.5)(3.0,1.5)(2.0,2.0)(1,1.5)(0,2)
(-1,1.5)(-1,0.5)(0,0)
\psline[linestyle=solid,linecolor=black]%
(0,0)(1,0.5)(2.0,0.0)(3.0,0.5)(3.0,1.5)(2.0,2.0)(1,1.5)(0,2)
(-1,1.5)(-1,0.5)(0,0)
\psdots[dotstyle=o,dotsize=2.2pt,linewidth=1.2pt,linecolor=black,fillcolor=gray!80]%
(0,0)(1,0.5)(2.0,0.0)(3.0,0.5)(3.0,1.5)(2.0,2.0)(1,1.5)(0,2)
(-1,1.5)(-1,0.5)(0,0)
\psline*[linestyle=solid,linecolor=white]%
(0,0.25)(1,0.75)(2.0,0.25)(2.5,0.5)(2.5,1.25)(2.0,1.55)(1,1.25)(0,1.5)
(-.55,1.25)(-.55,0.75)(0,0.25)
\psline[linestyle=solid,linecolor=black]%
(0,0.25)(1,0.75)(2.0,0.25)(2.5,0.5)(2.5,1.25)(2.0,1.55)(1,1.25)(0,1.5)
(-.55,1.25)(-.55,0.75)(0,0.25)
\psdots[dotstyle=o,dotsize=2.2pt,linewidth=1.2pt,linecolor=black,fillcolor=gray!80]%
(0,0.25)(1,0.75)(2.0,0.25)(2.5,0.5)(2.5,1.25)(2.0,1.55)(1,1.25)(0,1.5)
(-.55,1.25)(-.55,0.75)(0,0.25)
\psline*[linestyle=solid,linecolor=orange!35]%
(2,2)(1,2.25)(1.0,3.55)(2.25,3.85)(3.5,3.25)(3.5,2.25)(2,2)
\psline[linestyle=solid,linecolor=black]%
(2,2)(1,2.25)(1.0,3.55)(2.25,3.85)(3.5,3.25)(3.5,2.25)(2,2)
\psdots[dotstyle=o,dotsize=2.2pt,linewidth=1.2pt,linecolor=black,fillcolor=gray!80]%
(2,2)(1,2.25)(1.0,3.55)(2.25,3.85)(3.5,3.25)(3.5,2.25)(2,2)
\psdots[dotstyle=o,dotsize=2.5pt,linewidth=1.2pt,linecolor=black,fillcolor=red!80]%
(2,2)
\psline*[linestyle=solid,linecolor=white]%
(2,2.25)(1.25,2.5)(1.25,3.0)(2.25,3.25)(3.25,3.0)(3.25,2.35)(2,2.25)
\psline[linestyle=solid,linecolor=black]%
(2,2.25)(1.25,2.5)(1.25,3.0)(2.25,3.25)(3.25,3.0)(3.25,2.35)(2,2.25)
\psdots[dotstyle=o,dotsize=2.2pt,linewidth=1.2pt,linecolor=black,fillcolor=gray!80]%
(2,2.25)(1.25,2.5)(1.25,3.0)(2.25,3.25)(3.25,3.0)(3.25,2.35)(2,2.25)
\psdots[dotstyle=o,dotsize=5.2pt,linewidth=1.2pt,linecolor=black,fillcolor=gray!90]%
(-.80,1.05)(2.80,0.85)(2.80,1.2)(2.30,3.41)(2.50,3.61)(2.80,3.31)
\psline*[linestyle=solid,linecolor=brown!10]%
(2,2)(1,2.0)(0.0,3.0)(-0.20,3.25)(-1.0,3.25)(-1.0,2.25)
(0.0,2.15)(1.0,1.75)
(2,2)
\psline[linestyle=solid,linecolor=black]%
(2,2)(1,2.0)(0.0,3.0)(-0.20,3.25)(-1.0,3.25)(-1.0,2.25)
(0.0,2.15)(1.0,1.75)
(2,2)
\psdots[dotstyle=o,dotsize=2.2pt,linewidth=1.2pt,linecolor=black,fillcolor=gray!80]%
(2,2)(1,2.0)(0.0,3.0)(-0.20,3.25)(-1.0,3.25)(-1.0,2.25)
(0.0,2.15)(1.0,1.75)
(2,2)
\psdots[dotstyle=o,dotsize=2.5pt,linewidth=1.2pt,linecolor=black,fillcolor=red!80]%
(2,2)
\psline*[linestyle=solid,linecolor=white]%
(-0.85,2.75)(-0.85,2.35)(0.0,2.35)(0.25,2.15)(0.25,2.45)
(-0.85,2.75)
\psline[linestyle=solid,linecolor=black]%
(-0.85,2.75)(-0.85,2.35)(0.0,2.35)(0.25,2.15)(0.25,2.45)
(-0.85,2.75)
\psdots[dotstyle=o,dotsize=2.2pt,linewidth=1.2pt,linecolor=black,fillcolor=gray!80]%
(-0.85,2.75)(-0.85,2.35)(0.0,2.35)(0.25,2.15)(0.25,2.45)
(-0.85,2.75)
\psline*[linestyle=solid,linecolor=brown!10]%
(6,1.75)(6,2.0)(5.0,3.25)(4.5,3.25)(4.5,1.75)(6,1.75)
\psline[linestyle=solid,linecolor=black]%
(6,1.75)(6,2.0)(5.0,3.25)(4.5,3.25)(4.5,1.75)(6,1.75)
\psdots[dotstyle=o,dotsize=2.2pt,linewidth=1.2pt,linecolor=black,fillcolor=gray!80]%
(6,1.75)(6,2.0)(5.0,3.25)(4.5,3.25)(4.5,1.75)(6,1.75)
\psline*[linestyle=solid,linecolor=white]%
(5.75,2.0)(5,2.75)(4.75,2.75)(4.75,2.0)(5.75,2.0)
\psline[linestyle=solid,linecolor=black]%
(5.75,2.0)(5,2.75)(4.75,2.75)(4.75,2.0)(5.75,2.0)
\psdots[dotstyle=o,dotsize=2.2pt,linewidth=1.2pt,linecolor=black,fillcolor=gray!80]%
(5.75,2.0)(5,2.75)(4.75,2.75)(4.75,2.0)(5.75,2.0)
\psline*[linestyle=solid,linecolor=green!10]%
(7,0.25)(7,1.25)(4.0,1.25)(4.0,0.25)(7,0.25)
\psline[linestyle=solid,linecolor=black]%
(7,0.25)(7,1.25)(4.0,1.25)(4.0,0.25)(7,0.25)
\psdots[dotstyle=o,dotsize=2.2pt,linewidth=1.2pt,linecolor=black,fillcolor=gray!80]%
(7,0.25)(7,1.25)(4.0,1.25)(4.0,0.25)(7,0.25)
\psline*[linestyle=solid,linecolor=white]%
(6.5,0.35)(6.25,1.0)(5.5,0.75)(4.5,1.0)(4.5,0.35)(5.5,0.45)(6.5,0.35)
\psline[linestyle=solid,linecolor=black]%
(6.5,0.35)(6.25,1.0)(5.5,0.75)(4.5,1.0)(4.5,0.35)(5.5,0.45)(6.5,0.35)
\psdots[dotstyle=o,dotsize=2.2pt,linewidth=1.2pt,linecolor=black,fillcolor=gray!80]%
(6.5,0.35)(6.25,1.0)(5.5,0.75)(4.5,1.0)(4.5,0.35)(5.5,0.45)(6.5,0.35)
\rput(-0.85,3.75){\footnotesize $\boldsymbol{rbx K}$}
\rput(0.0,1.75){\footnotesize $\boldsymbol{rb A}$}
\rput(5.5,1.05){\footnotesize $\boldsymbol{rb A'}$}
\rput(0.3,3.5){\footnotesize $\boldsymbol{rbNrv K}$}
\rput(2.0,1.85){\footnotesize $\boldsymbol{a}$}
\rput(1.5,3.35){\footnotesize $\boldsymbol{rb B}$}
\rput(-0.50,3.00){\footnotesize $\boldsymbol{rb B'}$}
\rput(4.85,2.95){\footnotesize $\boldsymbol{rb B''}$}
\end{pspicture}
\caption[]{Sample ribbon complex $\rbx K$}
\label{fig:ribbonComplex}
\end{figure}

\begin{example} 
A planar ribbon $\rb E$ is shown in Fig.~\ref{fig:rbE} on a pair of nested filled cycles $\cyc A, \cyc B$ with cycle $\cyc B$ in the interior of cycle $\cyc A$ and the interior $\Int(\cl(\cyc B))$ is removed (not included) in the interior of cycle $\cyc A$. The white region in the interior of $\cyc A$ in Fig.~\ref{fig:rbE} represents the interior of $\cyc B$ not included in $\Int(\cyc A)$.  
\textcolor{blue}{\Squaresteel}
\end{example} 

\begin{definition}\label{def:ribbonComplex} \emph{Vergili Ribbon Complex}\\ 
Let $2^{VertK}$ denote the collection of subsets of vertices in a CW space $K$.  A Vergili\footnote{The structure of a ribbon complex was suggested by T. Vergili~\cite{Vergili2019personalComm}.} ribbon complex $K$ (denoted by $\rbx K$) is a non-void collection of countable planar ribbons in a CW space, {\em i.e.},
\[
\rbx K = \left\{\rb E:E\in 2^{VertK}\right\}.\mbox{\textcolor{blue}{\Squaresteel}}
\]
\end{definition}

\begin{example} 
Examples of Vergili ribbon complexes are given
in
\begin{compactenum}[1$^o$]
\item Fig.~\ref{fig:rbE}:
$\rbx K = \left\{\rb E\right\}$
\item Fig.~\ref{fig:rbNrvE}:
$\rbx K = \left\{\rb A,\rb B\right\}$
\item Fig.~\ref{fig:ribbonComplex}: 
$\rbx K = \left\{\rb A,\rb B',\rb B,\rb A',\rb B''\right\}$
\textcolor{blue}{\Squaresteel}
\end{compactenum} 
\end{example} 

A \emph{filled planar cycle} $A$ (denoted by $\cyc A$) has a nonempty interior with a boundary containing a non-void finite, collection $E$ of path-connected vertices so that there is a path between any pair of vertices in $\cyc A$.  The interior of the inner cycle in a ribbon is excluded from the ribbon. The outer and inner boundaries of a ribbon are simple, closed curves. A \emph{simple curve} has no self-intersections (loops).  A \emph{closed curve} begins and ends in the same vertex for each vertex in the curve.  A \emph{filled planar cycle} includes that part of the plane inside the cycle boundary.  A pair of cycles are \emph{nesting}, provided one cycle contains the other cycle entirely within its interior.  Let $\rb E$ be a planar ribbon.
The closure of a ribbon $E$ (denoted by $\cl(\rb E)$) includes its boundary (denoted by $\bdy(\rb E)$) and its interior (denoted by $\Int(\rb E)$).  The boundary of a filled cycle $\cyc A$ with the cycle interior excluded is denoted by $\cl(\cyc A)\setminus \Int(\cyc A)$. 

\begin{remark}
A planar ribbon is analogous to a Brooks ribbon~\cite[\S 1.2]{Brooks2013wideRibbons}, whose boundaries are a pair of simple, smoothly closed curves.  A ribbon structure is also analogous to a vortex tube, which is a collection of lines that form a vortex surface or vector tube~\cite[\S 1.3, p. 7]{Cottet2000CUPvortexMethods}.  An important application of planar ribbons appears in recent work on two-ribbon solar flares by H. He and others~\cite{He2019twoRibbonSolarFlares}. Although not considered here, it is possible to represent a non-void collection of convex ribbons in a ribbon complex as a Klee-Phelps convex groupoid~\cite{Peters2017convexGroupoids}.
\qquad \textcolor{blue}{\Squaresteel}
\end{remark}

A planar ribbon divides the plane into three disjoint open sets and provides a boundary of each of the three planar regions. (see Theorem~\ref{thm:BrouwerOpenSets}).  In this paper, the focus is on ribbons on a finite, bounded region of the plane.  In that case, a ribbon in the interior of a finite, bounded region of the plane divides the region into three disjoint bounded regions (see Theorem~\ref{thm:ribbonPlanarRegions}).  In addition, a continuous map on a planar ribbon to itself has a fixed point (see the Ribbon Fixed Point Theorem~\ref{thm:ribbonFixedPoint}), which is a straightforward consequence of the Brouwer Fixed Point Theorem.

\begin{definition}\label{def:ribbonNerve}
A \emph{ribbon nerve} $E$ (denoted by $\rbNrv K$) is a non-void collection of planar ribbons (in a ribbon complex $\rbx K$) that have nonempty intersection, {\em i.e.},
\[
\rbNrv K = \left\{\rb E\in \rbx K: \bigcap\rb E\neq\emptyset\right\}.\mbox{\qquad \textcolor{blue}{\Squaresteel}}
\]
\end{definition}


\begin{example} 
Examples of ribbon nerves derived from a ribbon complex $\rbx K$ on a finite bounded planar region in a CW space are given in
\begin{compactenum}[1$^o$]
\item Fig.~\ref{fig:rbE}:
$\rbNrv K = \left\{\{\rb E\}\right\}$
\item Fig.~\ref{fig:rbNrvE}:
$\rbNrv K = \left\{\{\rb A,\rb B\}\right\}$
\item Fig.~\ref{fig:ribbonComplex}: 
$\rbNrv K = \left\{\{\rb A,\rb B',\rb B\},\{\rb A'\},\{\rb B''\}\right\}$
\textcolor{blue}{\Squaresteel}
\end{compactenum} 
\end{example} 
   
Planar ribbon nerve complexes are examples of Edelsbrunner-Harer nerves.

\begin{definition}\label{def:EdelsbrunnerHarerNerve}
Let $F$ be a finite, non-void collection of sets.   An \emph{\bf Edelsbrunner-Harer nerve}~\cite[\S III.2, p. 59]{Edelsbrunner1999} consists of all nonempty subcollections of $F$ (denoted by $\Nrv F$) whose sets have nonempty intersection, {\em i.e.},
\[
\Nrv F = \left\{X\subseteq F: \bigcap X\neq \emptyset\right\}.\mbox{\qquad \textcolor{blue}{\Squaresteel}} 
\]
\end{definition}

\begin{theorem}
A ribbon nerve is an Edelsbrunner-Harer nerve.
\end{theorem}
\begin{proof}
Let $F$ be a finite collection of ribbons in a CW space.  Let the ribbon nerve $K$ (denoted by $\rbNrv K$) be defined by
\[
\rbNrv K = \overbrace{\left\{\rb E\in F:\bigcap \rb E \neq\emptyset\right\}.}^{\mbox{\textcolor{blue}{\bf Planar Ribbons in $F$ that have a common part}}}
\] 
Hence, from Def.~\ref{def:EdelsbrunnerHarerNerve}, $\rbNrv K$ is an Edelsbrunner-Harer nerve.
\end{proof}


A partially filled planar ribbon interior contains planar holes.  A \emph{planar hole} is a finite planar region bounded by a simple, closed curve that has an empty interior and cannot be retracted (shrunk) to a point.  That is, the interior of a planar hole contains no cells.  Holes in ribbons are analogous to surface jump discontinuities (gaps in a surface map) commonly found in vortex structures~\cite{Krasny1991jumpDiscontinuities}.

\begin{example}
Ribbon $\rb E$ in Fig.~\ref{fig:vortexCycles}(1.1) contains two holes in the interior of cycle $\cyc A$.  In each case, a hole is represented by an opaque grey region.   Again, for example, ribbon $\rb B$ in Fig.~\ref{fig:vortexCycles}(1.2) contains three holes (opaque gray regions) in its interior.  
\qquad \textcolor{blue}{\Squaresteel}
\end{example} 

Ribbons in a CW complex can be extracted from ordinary vortex nerves~\cite{Peters2019AMSBullVortexNerves}.

\begin{definition}\label{def:vortexNerve}
Let $K$ be a finite CW complex and let $2^{K}$ be the collection of all subsets of cells in $K$. A \emph{\bf vortex nerve} consists of a nonempty collection $E$ of nesting, usually non-concentric filled cycles $\cyc A$ in $K$ (denoted by $\vNrv K$) that have have nonempty intersection and which have zero or more edges (called {\rm \emph{filaments}}) attached between pairs of cycles in $\vNrv K$, {\em i.e.},
\[
\vNrv K = \left\{\cyc A\in 2^{K}: \bigcap \cyc A\neq \emptyset\right\}. \mbox{\qquad \textcolor{blue}{\Squaresteel}}
\]
\end{definition}

\begin{example}
A collection $X$ of filled cycles = $\left\{\cyc A, \cyc A', \cyc B\right\}$ in a CW space $K$ is represented in Fig.~\ref{fig:nerveRibbons}.  In this case, we have
\[
\vNrv K = \left\{\cyc A\in X: \bigcap \cyc A\neq \emptyset\right\}. \mbox{\qquad \textcolor{blue}{\Squaresteel}}
\]
That is, the intersection of all cycles in the collection $X$ is nonempty   Hence, $\vNrv K$ in Fig.~\ref{fig:nerveRibbons} is an Edelsbrunner-Harer nerve. 
\qquad \textcolor{blue}{\Squaresteel}
\end{example}

\begin{theorem}
A ribbon is a vortex nerve.
\end{theorem}
\begin{proof}
Let $\rb E$ be a ribbon in planar CW space $K$ containing a pair of nesting, non-concentric cycles $\cyc A, cyc B$ such that the boundary of $\cyc B$ is in the interior of $\cyc A$ and the interior of $\cyc B$ is not included in the interior of $\cyc A$, {\em i.e.}, $\Int(\cl(\cyc A))\supset \bdy(\cl(\cyc B)\setminus \Int(\cyc B))$.  
Consequently, $\cyc A \cap \cyc B\neq \emptyset$.  Then define nerve $\Nrv$ to be
\[
\Nrv K = \left\{\cyc A\in \rb E: \bigcap\cyc A \neq \emptyset\right\}.
\]
Hence, $\Nrv K$ is a vortex nerve.
\end{proof}

\begin{figure}[!ht]
\centering
\begin{pspicture}
(-1.5,-0.5)(4.0,4.0)
\psframe[linewidth=0.75pt,linearc=0.25,cornersize=absolute,linecolor=blue](-1.25,-0.25)(3.25,4)
\psline*[linestyle=solid,linecolor=green!30]%
(0,0)(1,0.5)(2.0,0.0)(3.0,0.5)(3.0,1.8)(2.0,2.0)(1,2.5)(0,2.25)
(-1,1.5)(-1,0.5)(0,0)
\psline[linestyle=solid,linecolor=black]%
(0,0)(1,0.5)(2.0,0.0)(3.0,0.5)(3.0,1.8)(2.0,2.0)(1,2.5)(0,2.25)
(-1,1.5)(-1,0.5)(0,0)
\psdots[dotstyle=o,dotsize=2.2pt,linewidth=1.2pt,linecolor=black,fillcolor=gray!80]%
(0,0)(1,0.5)(2.0,0.0)(3.0,0.5)(3.0,1.8)(2.0,2.0)(1,2.5)(0,2.25)
(-1,1.5)(-1,0.5)(0,0)
\psline*[linestyle=solid,linecolor=white]%
(0,0.25)(1,0.75)(2.0,0.25)(2.5,0.5)(2.5,1.0)(2.0,1.35)(1,1.55)(0,1.85)
(-.55,1.25)(-.55,0.75)(0,0.25)
\psline*[linestyle=solid,linecolor=black]%
(0,0.25)(1,0.75)(2.0,0.25)(2.5,0.5)(2.5,1.0)(2.0,1.35)(1,1.55)(0,1.85)
(-.55,1.25)(-.55,0.75)(0,0.25)
\psdots[dotstyle=o,dotsize=2.2pt,linewidth=1.2pt,linecolor=black,fillcolor=gray!80]%
(0,0.25)(1,0.75)(2.0,0.25)(2.5,0.5)(2.5,1.0)(2.0,1.35)(1,1.55)(0,1.85)
(-.55,1.25)(-.55,0.75)(0,0.25)
\psline*[linestyle=solid,linecolor=white]%
(0,0.35)(1,0.85)(2.0,0.35)(2.3,0.65)(2.3,0.85)(2.0,1.05)(1,1.15)(0,1.25)
(-.35,1.00)(-.35,0.75)(0,0.35)
\psline[linestyle=solid,linecolor=black]%
(0,0.35)(1,0.85)(2.0,0.35)(2.3,0.65)(2.3,0.85)(2.0,1.05)(1,1.15)(0,1.25)
(-.35,1.00)(-.35,0.75)(0,0.35)
\psdots[dotstyle=o,dotsize=2.2pt,linewidth=1.2pt,linecolor=black,fillcolor=gray!90]%
(0,0.35)(1,0.85)(2.0,0.35)(2.3,0.65)(2.3,0.85)(2.0,1.05)(1,1.15)(0,1.25)
(-.35,1.00)(-.35,0.75)(0,0.35)
\rput(-1.0,3.75){\footnotesize $\boldsymbol{K}$}
\rput(1.00,2.70){\footnotesize $\boldsymbol{vNrv K}$}
\rput(0.4,1.98){\footnotesize $\boldsymbol{rb E'}$}
\rput(0.0,1.50){\footnotesize \textcolor{white}{$\boldsymbol{rb E}$}}
\rput(2.8,2.05){\footnotesize $\boldsymbol{cyc B}$}
\rput(2.58,1.30){\footnotesize $\boldsymbol{cyc A'}$}
\rput(1.78,0.85){\footnotesize $\boldsymbol{cyc A}$}s
\end{pspicture}
\caption[]{Vortex nerve}
\label{fig:nerveRibbons}
\end{figure}

In general, a vortex nerve $\vNrv K$ contains $k$ nesting, non-concentric cycles.  The number of cycles $k$ in $\vNrv E$ can be either an even or odd number. By definition, each pair of cycles in $\vNrv K$ closest to each other is a ribbon complex.  It is always the case that ribbons $\rb A,\rb A'$ in ascending order and next to each other in nerve $\vNrv E$ have a common cycle, {\em i.e.}, the outer boundary $\bdy(\rb A')$ is also the inner boundary of $\rb A$.  These observations lead to the following result.

\begin{theorem}\label{thm:vortexNerveRibbons}
A vortex nerve with $k>1$ cycles contains $k-1$ ribbons.
\end{theorem}

\begin{example}
A sample vortex nerve $\vNrv K$ containing 3 nesting, non-centric cycles $\cyc A$,\\ $\cyc A'$,$\cyc B$ (in ascending order) and 2 ribbons $\rb E, \rb E'$ (also in ascending order) is represented in
Fig.~\ref{fig:nerveRibbons}.  In this particular nerve, cycle $\cyc A'$ is the outer boundary of ribbon
$\rb $ and the inner boundary of ribbon $\rb E'$.
\qquad \textcolor{blue}{\Squaresteel}
\end{example}

\begin{theorem}\label{thm:vortexNerveRibbonNerves}
A vortex nerve with $k>2$ cycles contains $k-2$ ribbon nerves.
\end{theorem}
\begin{proof}
Immediate from Theorem~\ref{thm:vortexNerveRibbons} and Def.~\ref{def:ribbonNerve} (ribbon nerve).
\end{proof} 

\begin{example}
From Theorem~\ref{thm:vortexNerveRibbonNerves}, a vortex nerve $\vNrv K$ containing $k = 3$ cycles has $k-2 = 1$ ribbon nerve $\rbNrv K = \left\{\rb E, \rb E'\in \vNrv K:\rb E\ \cap\ \rb E' \neq \emptyset\right\}$, which is represented in Fig.~\ref{fig:nerveRibbons}.
\qquad \textcolor{blue}{\Squaresteel}
\end{example}

\section{Preliminaries}
This section briefly introduces the \emph{approximate} descriptive closeness of cell complexes in a CW space.  For the axioms for an approximate descriptive proximity of nonempty sets $A,B$ (denoted by $A\ \dxnear\ B$), the usual set intersection $\cap$ for a traditional spatial proximity (denoted by $\delta$)~\cite[\S 1, p. 7]{Naimpally70},~\cite[p. 538]{Peters2012ams} is replaced by an approximate descriptive intersection (denoted by $\dxcap$), introduced in~\cite[\S 7.2, p. 303]{Peters2020CGTPhysics}, which is an extension of ordinary descriptive proximity (denoted by $\dnear$)~\cite{DiConcilio2018MCSdescriptiveProximities}, defined with descriptive intersection (denoted by $\dcap$) with a number of applications such as~\cite{Inan2019TJMdescriptiveProximity,AhmadPeters2018centroidalVortices,Peters2020CGTPhysics}.

Approximate closeness of nonempty sets $A,B$ is measured in terms of the Euclidean distance between feature vectors $\boldsymbol{\vec{a},\vec{b}}$ (denoted by $\norm{\boldsymbol{\vec{a}}-\boldsymbol{\vec{b}}}$) in $n$-dimensional Euclidean space $\mathbb{R}^n$.  In this context, $\mathbb{R}^n$  is a feature space in which each feature vector is a description of a nonempty set in a space $X$.  Let $2^X$ denote the collection of subsets in $X$. A probe function $\Phi:2^X\longrightarrow \mathbb{R}^n$ maps each nonempty subset $A$ in $2^X$ to a feature vector that describes $A$.      
The mapping $\Phi:2^X\longrightarrow \mathbb{R}^n$ is defined by
\[
\Phi(A) = \overbrace{\left(\mathbb{R}_1,\dots,\mathbb{R}_n\right).}^{\mbox{\textcolor{blue}{\bf Feature vector that describes $A\in 2^X$ in Euclidean space $ \mathbb{R}^n$}}}
\]

Nonempty sets $A,B$ with overlapping descriptions are descriptively proximal (denoted by $A\ \dnear\ B$). 
  The descriptive intersection\footnote{Many thanks to Tane 
	Vergili for suggesting a refinement of the usual definition for descriptive proximity in, {\em e.g.},~\cite[\S 3, p. 95]{DiConcilio2018MCSdescriptiveProximities}.} of nonempty subsets in $A\cup B$ (denoted by $A\ \dcap\ B$)  is defined by
\[
A\ \dcap\ B = \overbrace{\left\{x\in A\cup B: \Phi(x) \in \Phi(A)\ \cap\ \Phi(B)\right\}.}^{\mbox{\textcolor{blue}{\bf {\em i.e.}, $\boldsymbol{\mbox{Descriptions}\ \Phi(A)\ \&\ \Phi(B)\ \mbox{overlap}}$}}}
\] 
This form of descriptive proximity is not very useful, since it often the case that nonempty sets $A,B$ are close descriptively and yet $\Phi(A)\neq \Phi(B)$.  For this reason, we consider an approximate descriptive intersection on ribbon complexes.   In this form of set intersection, a pair of ribbon complexes $\rbx A$ on cell complex $K$ and $\rbx B$ on cell complex $K'$ belong to the collection of ribbons that have nonempty approximate descriptive intersection, provided the norm of the difference between the feature vectors $\Phi(\rbx A), \Phi(\rbx B)$ is less than some chosen threshold $th > 0$, {\em i.e.},
\begin{align*}
\mbox{\bf approximation threshold:}\ th &> 0.\\
\rbx A \in K, &\rbx B \in K'\\
\rbx A\ \dxnear\ \rbx B &\ \mbox{provided}\  \norm{\Phi(\rbx A) - \Phi(\rbx B)} < th:\\
                                 & \overbrace{\rbx A, \rbx B\in K\ \dxcap\ K'.}^{\mbox{\textcolor{blue}{\bf Descriptive intersection of ribbon complexes}}}
\end{align*}

\noindent In effect, we have
\[
K\ \dxcap\ K' = \overbrace{\left\{\rbx A, \rbx B\in K\cup K': \norm{\Phi(rbx A) - \Phi(rbx B)} < th\right\}.}^{\mbox{\textcolor{blue}{\bf Ribbon complexes that are $\boldsymbol{\dxnear}$ close}}}
\]


\noindent This leads to two possible forms of approximate descriptive closeness to consider, namely,
\begin{compactenum}[1$^o$] 
\item Approximate descriptive closeness of ribbon complexes $\rbx E\in K, \rbx E'\in K$ on a single space $K$, defined by
\[
\rbx A\ \dxnear\ \rbx B = \overbrace{\left\{\rbx E, \textbf{\cx} E'\in K: \norm{\Phi(\rbx E) - \Phi(\rbx E')} < th \right\},}^{\mbox{\textcolor{blue}{\bf Descriptions $\boldsymbol{\Phi(\rbx E),\Phi(\rbx E')}$ are\ $\dxnear$\ close in $K$ for $th$>0}}}
\]
\item Approximate descriptive closeness of ribbon complexes $\rbx E, \rbx E'$ on separated (disjoint) spaces $K, K'$, defined by
\[
\rbx A\ \dxnear\ \rbx B = \overbrace{\left\{\rbx E, \rbx E'\in K\cup K': \norm{\Phi(\rbx E) - \Phi(\rbx E')} < th \right\}.}^{\mbox{\textcolor{blue}{\bf Descriptions $\boldsymbol{\Phi(\rbx E),\Phi(\rbx E')}$ are\ $\dxnear$\ close in $K\cup K'$ for $th$>0}}}
\]
\end{compactenum}

An approximate descriptive intersection of distinct CW complexes $K, K'$  is easily derived from ordinary descriptive intersection $\dcap$ by considering the norm of the difference between the feature vectors $\Phi(\cx A), \Phi(\cx B)$, which is less than some chosen threshold $th > 0$, {\em i.e.},
\begin{align*}
\mbox{\bf\qquad Approximation threshold:}\ th &> 0.\\
\cx A \in K, &\cx B \in K'\\
\cx A\ \dxnear\ \cx B &,\ \mbox{provided}\  \norm{\Phi(\cx A) - \Phi(\cx B)} < th:\\
                                 &K\ \dxcap\ K'=\\ 
																 &\overbrace{\left\{\cx A, \cx B\in K\cup K':\norm{\Phi(\cx A) - \Phi(\cx B)} < th\right\}.}^{\mbox{\textcolor{blue}{\bf Approx. descriptive intersection of cell complexes}}}
\end{align*}

\begin{example}
Let $K,K'$ be CW complexes, ribbons $\rb E\in K, \rb A\in K'$ and let $\mathcal{B}_1$ be the Betti number, which is a count of the number of cycles in a cell complex, threshold $th = 1$ and let $\Phi(\rb E) = \mathcal{B}_1(\rb E), \Phi(\rb A) = \mathcal{B}_1(\rb A)$.  Then, for instance, we have
\begin{align*}
\mathcal{B}_1(\rb E\in K) &= 2\ \mbox{in Fig.~\ref{fig:rbE}}.\\
\mathcal{B}_1(\rb A\in K') &= 2\ \mbox{in Fig.~\ref{fig:rbNrvE}}.\ \mbox{Hence,}\\
\rb E\ \dxnear\ \rb A &,\ \mbox{since}\ \norm{\Phi(\rb E) - \Phi(\rb A)} < th,\ \mbox{and}\\
K\ \dxcap\ K' &\neq \emptyset\ \mbox{\qquad \textcolor{blue}{\Squaresteel}}.
\end{align*}
\end{example}

\noindent Many other instances of non-void approximate descriptive intersection are possible.

\begin{example}
Let $K,K'$ be a pair of cell complexes on a finite, bounded planar region.  Then consider sub-complexes $\cx A\in K, \cx B\in K'$ that are close to each other for some threshold $th > 0$.  That is,
\[
K\ \dxcap\ K' = \overbrace{\left\{\cx A, \cx B\in K\cup K': \norm{\Phi(\cx A) - \Phi(\cx B)} < th\right\}.}^{\mbox{\textcolor{blue}{\bf Collection of cell complexes that are $\boldsymbol{\dxnear}$ close}}}\mbox{\qquad \textcolor{blue}{\Squaresteel}}
\]
\end{example}

\begin{example}
Consider isolating those ribbon complexes $\rb A\in K, \rb B\in K'$ that are close to each other within some chosen threshold $th >0$, {\em i.e.},
\[
K\ \dxcap\ K' = \overbrace{\left\{\rb A\in K, \rb B\in K': \norm{\Phi(\rb A) - \Phi(\rb B)} < th\right\}.}^{\mbox{\textcolor{blue}{\bf Ribbon shapes that are $\boldsymbol{\dxnear}$ close}}}
\]
Let $\mathcal{B}_2(\cx E)$ be the Betti number which is a count of the number of holes in a cell complex $\cx E$.  In Fig.~\ref{fig:rbE}, $\mathcal{B}_2(\rb E)$ = 2, {\em i.e.}, only ribbon $\rb E$ has two holes, which appear as opaque regions {\textcolor{gray!80}{\Large $\boldsymbol{\bullet}$}} on $\rb E$.  In Fig.~\ref{fig:rbNrvE}, $\mathcal{B}_2(\rbNrv K)$ = 5, $\mathcal{B}_2(\rb A)$ = 2, and $\mathcal{B}_2(\rb B)$ = 3.  Let threshold $th = 1$, $\Phi(\cx E) = \mathcal{B}_2(\cx E)$ and require $\norm{\Phi(\cx A) - \Phi(\cx B)} < th$ for a pair of CW cell complexes $\cx A,\cx B$ in a CW complex $K$.  Then, we have
\begin{align*}
th &= 1.\\
\norm{\Phi(\rb E)) - \Phi(\rb A)} &= 0 < th\Rightarrow \rb E\ \dxnear\ \rb A.\\
\norm{\Phi(\rb E)) - \Phi(\rb B)} &\nless\ th\Rightarrow \rb E\ \not\dxnear\ \rb B.\\
\norm{\Phi(\rb E)) - \Phi(\rbNrv K)} &\nless\ th\Rightarrow \rb E\ \not\dxnear\ \rbNrv E.\\
\end{align*}
In effect, only $\rb E$ in Fig.~\ref{fig:rbE} and $\rb A$ in Fig.~\ref{fig:rbNrvE} have approximate descriptive proximity, {\em i.e.} only these ribbon complexes are approximately close descriptively relative to the chosen threshold and feature vector defined in terms of the Betti number $\mathcal{B}_2(\rb E)$ for a ribbon complex $\rb E$ in the CW space $K$ and $\mathcal{B}_2(\rb A)$ for a ribbon complex $\rb A$ in the CW space $K'$ represented in Fig.~\ref{fig:vortexCycles}.
\qquad \textcolor{blue}{\Squaresteel}
\end{example}

Let $K$ be a finite, bounded, planar nonempty space, $A,B\in 2^K$. If $A \dxcap B$ is nonempty, there is at least one element of $A$ with a description that approximately matches the description of an element of $B$.    It is entirely possible to identify a pair of nonempty sets $A, B$ separated spatially ({\em i.e.}, $A$ and $B$ have no members in common) and yet $A,B$ have approximate matching descriptions.    The pair $\left(K,\dxnear\right)$ is an approximate descriptive proximity space, provided the following axioms are satisfied.

\begin{description}
\item[{\rm\bf (xdP0)}] $\emptyset\ \not{\dxnear}\ A,\ \mbox{for all}\ A \subset K$, {\em i.e.}, the empty set is far from any cell complex in $K$.
\item[{\rm\bf (xdP1)}] $A\ \dxnear\ B \Leftrightarrow B\ \dxnear\ A$.
\item[{\rm\bf (xdP2)}] $A\ \dxcap\ B \neq \emptyset \Rightarrow\ A\ \dxnear\ B$.
\item[{\rm\bf (xdP3)}] $A\ \dxnear\ (B \cup C) \Leftrightarrow A\ \dxnear\ B $ or $A\ \dxnear\ C$.
\end{description}

\noindent The converse of axiom {\bf xdP2} also holds (see Lemma~\ref{lemma:dP2converse}).  To prove this, we introduce an approximate form of descriptive intersection of nonvoid sub-complexes $\cx E, \cx E'$ on $K$ (denoted by $\cx E\ \dxcap\ \cx E'$) is defined by
\[
\mathop{\bigcap}\limits_{\dxnear}K = \overbrace{\left\{\sh E, \sh E'\in 2^K: \norm{\Phi(\sh E) - \Phi(\sh E')} < th \right\},}^{\mbox{\textcolor{blue}{\bf Descriptions $\boldsymbol{\Phi(\sh E),\Phi(\sh E')}$ are close}}}.
\]

\begin{lemma}\label{lemma:dP2converse}
Let $K$ be a CW complex equipped with the relation $\dxnear$, $A,B\in 2^K$.   Then $A\ \dxnear\ B$ implies $A\ \dxcap\ B\neq \emptyset$.
\end{lemma}
\begin{proof}
Let $th > 0$, nonempty $A,B\in 2^K$. By definition, $A\ \dxnear\ B$ implies that $\norm{\Phi(A) - \Phi(B)} < th$.   Consequently, $A,B\in \mathop{\bigcap}\limits_{\xnear}K$.  Hence, $A\ \dxcap\ B \neq \emptyset$ and the converse of {\bf (xdP2)} follows.    
\end{proof}

\begin{theorem}\label{theorem:dP2result}
Let $K$ be a collection of planar ribbon complexes equipped with the proximity $\dxnear$, $\rb A,\rb B\in 2^K$.   Then $\rb A\ \dxnear\ \rb B$ implies $\rb A\ \dxcap\ \rb B\neq \emptyset$.
\end{theorem}
\begin{proof}
Immediate from Lemma~\ref{lemma:dP2converse}. 
\end{proof}

\begin{corollary}\label{cor:dP2result}
Let $K$ be a collection of planar ribbon nerves equipped with the proximity $\dxnear$, $\rbNrv A,\rbNrv B\in 2^K$.   Then $\rbNrv A\ \dxnear\ \rbNrv B$ if and only if $\rbNrv A\ \dxcap\ \rbNrv B\neq \emptyset$.
\end{corollary}
\begin{proof}$\mbox{}$\\
Let $K$ be equipped with $\dxnear$. \\
$\Rightarrow$:  From Lemma~\ref{lemma:dP2converse}, $\rbNrv A\ \dxnear\ \rbNrv B$ implies $\rbNrv A\ \dxcap\ \rbNrv B\neq \emptyset$.\\
$\Leftarrow$ From Axiom xdP2, $\rbNrv A\ \dxcap\ \rbNrv B\neq \emptyset$ implies $\rbNrv\ \dxnear\ \rbNrv$.
\end{proof}

\begin{lemma}\label{lemma:cellularProximitySpace}
Let $K$ be a collection of ribbon complexes equipped with the relation $\dxnear$, ribbons $\rb A,\rb B\in 2^K$.   Then $K$ is an approximate descriptive proximity space.
\end{lemma}
\begin{proof}$\mbox{}$\\
Let $K$ be a collection of ribbon complexes equipped with $\dxnear$ and threshold $th > 0$.   Then
(dP0): The empty set contains ribbons.   Hence, $\emptyset\ \not{\dxnear}\ \rb A, \forall\ \rb A \in 2^K$.\\
(dP1):  Assume $\rb A\ \dxnear\ \rb B$, if and only if $\norm{\Phi(\rb A) - \Phi(\rb B)} < th$.  Then $\rb A\ \dxcap\ \rb B \neq \emptyset$, if and only if $\rb B\ \dxcap\ \rb A \neq \emptyset$, if and only if $\rb B\ \dxnear\ \rb A$.\\
(dP2):  Assume $\rb A\ \dxcap\ \rb B \neq \emptyset$.  Then, by definition of $\dxcap$, 
$\norm{\Phi(\rb A) - \Phi(\rb B)} < th$.  Hence, $\rb A\ \dxnear\ \rb B$.\\
(dP3): $\rb A\ \dxnear\ (\rb B\ \cup\ \rb C)$, if and only if $\norm{\Phi(\rb A) - \Phi(\rb B)} < th$, implying $\rb A\ \dxnear\ \rb B$ or, by definition of $\cup$, $\norm{\Phi(\rb A) - \Phi(\rb C)} < th$, implying $\rb A\ \dxnear\ \rb C$.\\
Hence, $\left(K,\dxnear\right)$ is an approximate descriptive proximity space.
\end{proof}

\begin{theorem}\label{theorem:dP2vNRvResult}
Let $K$ be a collection of planar ribbon nerves equipped with the proximity $\dxnear$.   Then $\left(K,\dxnear\right)$ is an approximate descriptive proximity space.
\end{theorem}
\begin{proof}
Immediate from Lemma~\ref{lemma:cellularProximitySpace}, since each planar ribbon nerve is a collection of ribbons equipped with $\dxnear$. 
\end{proof}

\begin{figure}[!ht]
\centering
\subfigure[Ribbon nerve $\rbNrv E$ in a ribbon complex on CW space $K$]
 {\label{fig:rbNrvE1}
\begin{pspicture}
(-1.5,-0.5)(4.0,4.0)
\psframe[linewidth=0.75pt,linearc=0.25,cornersize=absolute,linecolor=blue](-1.25,-0.25)(3.75,4)
\psline*[linestyle=solid,linecolor=green!30]%
(0,0)(1,0.5)(2.0,0.0)(3.0,0.5)(3.0,1.5)(2.0,2.0)(1,1.5)(0,2)
(-1,1.5)(-1,0.5)(0,0)
\psline[linestyle=solid,linecolor=black]%
(0,0)(1,0.5)(2.0,0.0)(3.0,0.5)(3.0,1.5)(2.0,2.0)(1,1.5)(0,2)
(-1,1.5)(-1,0.5)(0,0)
\psdots[dotstyle=o,dotsize=2.2pt,linewidth=1.2pt,linecolor=black,fillcolor=gray!80]%
(0,0)(1,0.5)(2.0,0.0)(3.0,0.5)(3.0,1.5)(2.0,2.0)(1,1.5)(0,2)
(-1,1.5)(-1,0.5)(0,0)
\psline*[linestyle=solid,linecolor=white]%
(0,0.25)(1,0.75)(2.0,0.25)(2.5,0.5)(2.5,1.25)(2.0,1.55)(1,1.25)(0,1.5)
(-.55,1.25)(-.55,0.75)(0,0.25)
\psline[linestyle=solid,linecolor=black]%
(0,0.25)(1,0.75)(2.0,0.25)(2.5,0.5)(2.5,1.25)(2.0,1.55)(1,1.25)(0,1.5)
(-.55,1.25)(-.55,0.75)(0,0.25)
\psdots[dotstyle=o,dotsize=2.2pt,linewidth=1.2pt,linecolor=black,fillcolor=gray!80]%
(0,0.25)(1,0.75)(2.0,0.25)(2.5,0.5)(2.5,1.25)(2.0,1.55)(1,1.25)(0,1.5)
(-.55,1.25)(-.55,0.75)(0,0.25)
\psline*[linestyle=solid,linecolor=orange!35]%
(2,2)(1,2.25)(1.0,3.55)(2.25,3.85)(3.5,3.25)(3.5,2.25)(2,2)
\psline[linestyle=solid,linecolor=black]%
(2,2)(1,2.25)(1.0,3.55)(2.25,3.85)(3.5,3.25)(3.5,2.25)(2,2)
\psdots[dotstyle=o,dotsize=2.2pt,linewidth=1.2pt,linecolor=black,fillcolor=gray!80]%
(2,2)(1,2.25)(1.0,3.55)(2.25,3.85)(3.5,3.25)(3.5,2.25)(2,2)
\psdots[dotstyle=o,dotsize=2.5pt,linewidth=1.2pt,linecolor=black,fillcolor=red!80]%
(2,2)
\psline*[linestyle=solid,linecolor=white]%
(2,2.25)(1.25,2.5)(1.25,3.0)(2.25,3.25)(3.25,3.0)(3.25,2.35)(2,2.25)
\psline[linestyle=solid,linecolor=black]%
(2,2.25)(1.25,2.5)(1.25,3.0)(2.25,3.25)(3.25,3.0)(3.25,2.35)(2,2.25)
\psdots[dotstyle=o,dotsize=2.2pt,linewidth=1.2pt,linecolor=black,fillcolor=gray!80]%
(2,2.25)(1.25,2.5)(1.25,3.0)(2.25,3.25)(3.25,3.0)(3.25,2.35)(2,2.25)
\psdots[dotstyle=o,dotsize=5.2pt,linewidth=1.2pt,linecolor=black,fillcolor=gray!90]%
(-.80,1.05)(2.80,0.85)(2.80,1.2)(2.30,3.41)(2.50,3.61)(2.80,3.31)
\psline*[linestyle=solid,linecolor=brown!10]%
(2,2)(1,2.0)(0.0,3.0)(-0.20,3.25)(-1.0,3.25)(-1.0,2.25)
(0.0,2.15)(1.0,1.75)
(2,2)
\psline[linestyle=solid,linecolor=black]%
(2,2)(1,2.0)(0.0,3.0)(-0.20,3.25)(-1.0,3.25)(-1.0,2.25)
(0.0,2.15)(1.0,1.75)
(2,2)
\psdots[dotstyle=o,dotsize=2.2pt,linewidth=1.2pt,linecolor=black,fillcolor=gray!80]%
(2,2)(1,2.0)(0.0,3.0)(-0.20,3.25)(-1.0,3.25)(-1.0,2.25)
(0.0,2.15)(1.0,1.75)
(2,2)
\psdots[dotstyle=o,dotsize=2.5pt,linewidth=1.2pt,linecolor=black,fillcolor=red!80]%
(2,2)
\psline*[linestyle=solid,linecolor=white]%
(-0.85,2.75)(-0.85,2.35)(0.0,2.35)(0.25,2.15)(0.25,2.45)
(-0.85,2.75)
\psline[linestyle=solid,linecolor=black]%
(-0.85,2.75)(-0.85,2.35)(0.0,2.35)(0.25,2.15)(0.25,2.45)
(-0.85,2.75)
\psdots[dotstyle=o,dotsize=2.2pt,linewidth=1.2pt,linecolor=black,fillcolor=gray!80]%
(-0.85,2.75)(-0.85,2.35)(0.0,2.35)(0.25,2.15)(0.25,2.45)
(-0.85,2.75)
\rput(-1.0,3.75){\footnotesize $\boldsymbol{K}$}
\rput(0.0,1.75){\footnotesize $\boldsymbol{rb A}$}
\rput(0.3,3.5){\footnotesize $\boldsymbol{rbNrv E}$}
\rput(2.0,1.85){\footnotesize $\boldsymbol{a}$}
\rput(1.5,3.35){\footnotesize $\boldsymbol{rb B}$}
\rput(-0.50,3.00){\footnotesize $\boldsymbol{rb B'}$}
\end{pspicture}}\hfil
\subfigure[Ribbon nerve $\rbNrv E'$ in a ribbon complex on CW space $K'$]
 {\label{fig:rbNrvE2}
\begin{pspicture}
(-1.5,-0.5)(4.0,4.0)
\psframe[linewidth=0.75pt,linearc=0.25,cornersize=absolute,linecolor=blue](-1.25,-0.25)(3.75,4)
\psline*[linestyle=solid,linecolor=green!30]%
(0,0)(1,0.5)(2.0,0.0)(3.0,0.5)(3.0,1.5)(2.0,2.0)(1,1.5)(0,2)
(-1,1.3)(-1,0.5)(0,0)
\psline[linestyle=solid,linecolor=black]%
(0,0)(1,0.5)(2.0,0.0)(3.0,0.5)(3.0,1.5)(2.0,2.0)(1,1.5)(0,2)
(-1,1.3)(-1,0.5)(0,0)
\psdots[dotstyle=o,dotsize=2.2pt,linewidth=1.2pt,linecolor=black,fillcolor=gray!80]%
(0,0)(1,0.5)(2.0,0.0)(3.0,0.5)(3.0,1.5)(2.0,2.0)(1,1.5)(0,2)
(-1,1.3)(-1,0.5)(0,0)
\psline*[linestyle=solid,linecolor=white]%
(0,0.25)(1,0.75)(2.0,0.25)(2.5,0.5)(2.5,1.25)(2.0,1.55)(1,1.05)(0,1.5)
(-.55,1.25)(-.55,0.75)(0,0.25)
\psline[linestyle=solid,linecolor=black]%
(0,0.25)(1,0.75)(2.0,0.25)(2.5,0.5)(2.5,1.25)(2.0,1.55)(1,1.05)(0,1.5)
(-.55,1.25)(-.55,0.75)(0,0.25)
\psdots[dotstyle=o,dotsize=2.2pt,linewidth=1.2pt,linecolor=black,fillcolor=gray!80]%
(0,0.25)(1,0.75)(2.0,0.25)(2.5,0.5)(2.5,1.25)(2.0,1.55)(1,1.05)(0,1.5)
(-.55,1.25)(-.55,0.75)(0,0.25)
\psline*[linestyle=solid,linecolor=blue!20]%
(2,2)(1,1.50)(1.0,3.55)(2.25,3.85)(3.5,3.25)(3.5,2.25)(2,2)
\psline[linestyle=solid,linecolor=black]%
(2,2)(1,1.50)(1.0,3.55)(2.25,3.85)(3.5,3.25)(3.5,2.25)(2,2)
\psdots[dotstyle=o,dotsize=2.2pt,linewidth=1.2pt,linecolor=black,fillcolor=gray!80]%
(2,2)(1.0,3.55)(2.25,3.85)(3.5,3.25)(3.5,2.25)(2,2)
\psdots[dotstyle=o,dotsize=2.5pt,linewidth=1.2pt,linecolor=black,fillcolor=red!80]%
(1,1.50)(2,2)
\psline*[linestyle=solid,linecolor=white]%
(2,2.25)(1.25,2.5)(1.25,3.0)(2.25,3.25)(3.25,3.0)(3.25,2.35)(2,2.25)
\psline[linestyle=solid,linecolor=black]%
(2,2.25)(1.25,2.5)(1.25,3.0)(2.25,3.25)(3.25,3.0)(3.25,2.35)(2,2.25)
\psdots[dotstyle=o,dotsize=2.2pt,linewidth=1.2pt,linecolor=black,fillcolor=gray!80]%
(2,2.25)(1.25,2.5)(1.25,3.0)(2.25,3.25)(3.25,3.0)(3.25,2.35)(2,2.25)
\psdots[dotstyle=o,dotsize=5.2pt,linewidth=1.2pt,linecolor=black,fillcolor=gray!90]%
(-.80,1.05)(2.80,0.85)(2.80,1.2)(2.30,3.41)(2.50,3.61)(2.80,3.31)
\rput(-1.0,3.75){\footnotesize $\boldsymbol{K'}$}
\rput(0.0,1.75){\footnotesize $\boldsymbol{rb A'}$}
\rput(0.3,2.5){\footnotesize $\boldsymbol{rbNrv E'}$}
\rput(2.0,1.85){\footnotesize $\boldsymbol{a}$}
\rput(1,1.35){\footnotesize $\boldsymbol{a'}$}
\rput(1.5,3.35){\footnotesize $\boldsymbol{rb B'}$}
\end{pspicture}}
\caption[]{Ribbon nerves $\rbNrv E,\rbNrv E'$ on disjoint CW spaces $K,K'$, respectively}
\label{fig:separatedRibbonNerves}
\end{figure}

\begin{example}
Let $K$ in Fig.~\ref{fig:rbNrvE1} be a collection of planar ribbons equipped with $\dxnear$, forming a ribbon nerve $\rbNrv E$ in an approximate descriptive space $\left(K,\dxnear\right)$.  Similarly, let $K'$ in Fig.~\ref{fig:rbNrvE2} be a collection of planar ribbons equipped with $\dxnear$, forming a ribbon nerve $\rbNrv E'$ an approximate descriptive space $\left(K',\dxnear\right)$.  Also let $\left(K\cup K',\dxnear\right)$ be an approximate descriptive space.

Recall that Betti number $\mathcal{B}_2$ is a count of the number of holes in a cell complex.  Then let $\Phi(\rbNrv E) = \mathcal{B}_2$ and let $\Phi(\rbNrv E') = \mathcal{B}_2$,
{\em i.e.}, each of the ribbon nerves in Fig.~\ref{fig:separatedRibbonNerves} has description equal to the number of ribbon surface holes.
Ribbon nerve $\rbNrv E'$ is described in a similar fashion. Further, let threshold $th > 1$.  Then we have
\begin{align*}
\rbNrv E\ &\dxnear\ \rbNrv E',\ \mbox{since, for $th > 0$,}\\
\Phi(\rbNrv E) &= (\mathcal{B}_2(\rbNrv E)) = \Phi(\rbNrv E') = 6.
\end{align*}

That is, $\norm{\Phi(\rbNrv E) - \Phi(\rbNrv E')} < th$ for every choice of $th > 0$.
\qquad \textcolor{blue}{\Squaresteel}
\end{example}

\begin{theorem}\label{theorem:rbNrvResult2}
Let $K$ be a collection of planar ribbon nerves equipped with the proximity $\dxnear$, $\left(K,\dxnear\right)$ an approximate descriptive proximity space, $\rbNrv E, \rbNrv E'\in K, th > 0$.  $\rbNrv E\ \dxnear\ \rbNrv E'$, if and only if $\Phi(\rbNrv E) = \Phi(\rbNrv E')$.   
\end{theorem}
\begin{proof}
Immediate from the definition $\dxnear$ on pairs of ribbon nerves in $K$. 
\end{proof}

\section{Main Results}
This section gives some main results for ribbon complexes.

\subsection{Ribbon division of the plane into three bounded regions and Brouwer fixed points on ribbons}$\mbox{}$\\
L.E.J. Brouwer~\cite{Brouwer1910threeDisjointSurfaceRegions} introduced a curve which divides the plane into three open sets and provides a boundary of each of the regions.  An important result for a ribbon complex in a CW space is the division of a finite bounded region of the plane into three bounded regions.

\begin{theorem}\label{thm:ribbonPlanarRegions}
A ribbon in the interior of a finite, bounded region of the plane divides the region into three disjoint bounded regions.
\end{theorem}
\begin{proof}
Let $\rb E$ be a ribbon containing nesting cycles $\cyc A, \cyc B$ with $\bdy(\cl(\cyc B)) \subset \Int(\cl(\cyc A))$ on $\pi$, a finite, bounded region of the plane $\pi$. The planar region $\pi_1 = \pi\setminus \cl(\rb E)$ is that part of $\pi$ outside the ribbon.  A second region of the plane $\pi_2\subset \pi$ is the closure of $\rb E$ minus the closure of its inner cycle $\cl(\cyc B)$, defined by
\[ 
\pi_2 = \cl(\rb E)\setminus \cl(\cyc B). 
\]
$\pi_1\cap \pi_2 = \emptyset$, since ribbon $\rb E$ is not included in $\pi_1$.
A third region $\pi_3\subset \pi$ is the closure of $\cyc B$, {\em i.e.}, $\pi_3 = \cl(\cyc B)$.  $\pi_2\cap \pi_3 = \emptyset$, since $\cl(\cyc B)$ is not included in $\pi_2$.  Hence, these three planar regions are disjoint.
\end{proof}

It is also the case that a planar ribbon has the division property of the curve discovered by Brouwer.

\begin{theorem}\label{thm:BrouwerOpenSets}
A planar ribbon divides the plane into three open sets and provides a boundary of each of the three planar regions.
\end{theorem}
\begin{proof}
Let $\rb E$ be a ribbon on $\mathbb{R}^2$.  From Theorem~\ref{thm:ribbonPlanarRegions},
$\rb E$ divides the plane into three open sets and provides a boundary of each of the three planar regions.
\qquad \textcolor{blue}{\Squaresteel}
\end{proof}

\begin{remark} \emph{Approximate Descriptive Proximity of Brouwer Planar Regions}.\\  
Let $B = \left\{\pi_1,\pi_2,\pi_3\right\}$ be collection of planar regions from the proof of Theorem~\ref{thm:ribbonPlanarRegions}, equipped with the proximity $\dxnear$.  Also, let $\mathcal{B}_1$ be the Betti number with is a count of the number cycles in a cell complex and introduce threshold $0 < th \leq 1$.  For $E\in B$, let
$\Phi(E) = \mathcal{B}_1(E)$ and choose $th = 1$.  From this, we obtain
\begin{compactenum}[1$^o$]
\item $\pi_1\ \not{\dxnear}\ \pi_2$, since $\norm{\Phi(\pi_1)-\Phi(\pi_2)}\nless 1$.
\item $\pi_1\ \not{\dxnear}\ \pi_3$, since $\norm{\Phi(\pi_1)-\Phi(\pi_3)}\nless 1$.
\item $\pi_2\ \dxnear\ \pi_3$, since $\norm{\Phi(\pi_2)-\Phi(\pi_3)} < 1$.
\end{compactenum}
For $th > 1$, all three Brouwer planar regions from theorem~\ref{thm:ribbonPlanarRegions} do have approximate descriptive proximity.
\qquad \textcolor{blue}{\Squaresteel}
\end{remark}

\begin{figure}[!ht]
\centering
\begin{pspicture}
(3,2)
\psline*[linestyle=solid,linecolor=green!30]%
(0,0)(1,0.5)(2.0,0.0)(3.0,0.5)(3.0,1.5)(2.0,2.0)(1,1.5)(0,2)
(-1,1.5)(-1,0.5)(0,0)
\psline[linestyle=solid,linecolor=black]%
(0,0)(1,0.5)(2.0,0.0)(3.0,0.5)(3.0,1.5)(2.0,2.0)(1,1.5)(0,2)
(-1,1.5)(-1,0.5)(0,0)
\psdots[dotstyle=o,dotsize=2.2pt,linewidth=1.2pt,linecolor=black,fillcolor=gray!80]%
(0,0)(1,0.5)(2.0,0.0)(3.0,0.5)(3.0,1.5)(2.0,2.0)(1,1.5)(0,2)
(-1,1.5)(-1,0.5)(0,0)
\psdots[dotstyle=o,dotsize=5.2pt,linewidth=1.2pt,linecolor=black,fillcolor=gray!80]%
(-0.8,1.3)(-0.8,0.8)(0,1.8)
\psline*[linestyle=solid,linecolor=white]%
(0,0.25)(1,0.75)(2.0,0.25)(2.5,0.5)(2.5,1.0)(2.0,1.35)(1,1.25)(0,1.5)
(-.55,1.25)(-.55,0.75)(0,0.25)
\psline[linestyle=solid,linecolor=black]%
(0,0.25)(1,0.75)(2.0,0.25)(2.5,0.5)(2.5,1.0)(2.0,1.35)(1,1.25)(0,1.5)
(-.55,1.25)(-.55,0.75)(0,0.25)
\psline[linestyle=solid,linecolor=black]%
(1,0.75)(1,0.5)
\psdots[dotstyle=o,dotsize=2.2pt,linewidth=1.2pt,linecolor=black,fillcolor=gray!80]%
(0,0.25)(1,0.75)(2.0,0.25)(2.5,0.5)(2.5,1.0)(2.0,1.35)(1,1.25)(0,1.5)
(-.55,1.25)(-.55,0.75)(0,0.25)
\psdots[dotstyle=o,dotsize=2.5pt,linewidth=1.2pt,linecolor=black,fillcolor=red!80]%
(1,0.75)
\psdots[dotstyle=o,dotsize=2.5pt,linewidth=1.2pt,linecolor=black,fillcolor=black!80]%
(1,0.48)
\rput(1,0.9){\footnotesize $\boldsymbol{p}$}
\rput(1,0.3){\footnotesize $\boldsymbol{q}$}
\rput(1,1.75){\footnotesize $\boldsymbol{rb E}$}
\rput(2.8,1.85){\footnotesize $\boldsymbol{cyc A}$}
\rput(2.5,1.25){\footnotesize $\boldsymbol{cyc B}$}
\end{pspicture}
\caption[]{\textcolor{gray}{\large $\boldsymbol{\bullet}$} = ribbon hole, and
$
\overbrace{f(p\in \arc{pq}) = \arc{pq}\cap \bdy(\cl(\cyc B) = p).}^{\mbox{\textcolor{blue}{\bf Mapping from $p\in \arc{pq}$ to fixed point $p$ on ribbon $\rb E$}}}
$ 
}
\label{fig:ribbonFixedPoint}
\end{figure}

\begin{theorem}\label{thm:Brouwer}{Brouwer Fixed Point Theorem~\cite[\S 4.7, p. 194]{Spanier1966AlgTopology}}$\mbox{}$\\
Every continuous map from $\mathbb{R}^n$ to itself has a fixed point.
\end{theorem}

\begin{theorem}\label{thm:ribbonFixedPoint}{Ribbon Fixed Point Theorem\textsc{}}$\mbox{}$\\
A continuous map on a ribbon on $\mathbb{R}^2$ to itself has a fixed point.
\end{theorem}
\begin{proof}
Let $\rb E$ be a ribbon on $\mathbb{R}^2$. From Theorem~\ref{thm:Brouwer}, each continuous map $f:\rb E\longrightarrow \rb E$ has a fixed point.
\end{proof}

\begin{example}
Let $K\subset \mathbb{R}^2$ be a CW complex and let $\arc{pq}$ be an edge attached between a vertex $q$ on the outer $\bdy(\cl(\cyc A))$ and vertex $p$ on the inner $\bdy(\cl(\cyc B))$ on a planar ribbon $\rb E$ with cycles $\cyc A,\cyc B\in \rb E$, $\Int(\cyc A)\supset \bdy(\cyc B)$ as shown in Fig.~\ref{fig:ribbonFixedPoint}, {\em i.e.}, the boundary of cycle $\cyc B$ is a subset of the interior of cycle $\cyc A$.  
Then let the continuous map $f:\rb E\longrightarrow \rb E$ be defined by
\[
f(p\in \rb E) = \arc{pq}\cap \bdy(\cl(\cyc B).
\]
From Theorem~\ref{thm:ribbonFixedPoint}, the mapping $f$ has a fixed point. Let $\arc{pq}$ be an edge with vertex $p\in \bdy(\cl(\cyc B)), q\in \bdy(\cl(\cyc A))$ as shown in Fig.~\ref{fig:ribbonFixedPoint}.  The mapping $f$ maps vertex $p\in\arc{pq}$ to $\arc{pq}\cap \bdy(\cl(\cyc B) = p$, a fixed point on the boundary of $\cyc B$. That is, $f$ maps vertex $p\in \arc{pq}$ to $\arc{pq}\cap \bdy(\cl(\cyc B) = p$, which equals $p$.  Hence, $f(p\in \mathbb{R}^2) = p\in \mathbb{R}^2$, which is the desired result.
\qquad \textcolor{blue}{\Squaresteel}
\end{example}

The search for fixed points on ribbons, ribbon nerves and vortex nerves can be simplified as a consequence of Lemma~\ref{lemma:CWcomplex}, which is easily proved.

\begin{lemma}\label{lemma:CWcomplex}
A map from a nonempty CW complex to itself has a fixed point.
\end{lemma}

\begin{theorem}\label{thm:rbFixedPoint}$\mbox{}$\\
A map on a nonempty ribbon complex to itself has a fixed point.
\end{theorem}
\begin{proof}
Let $\rbx E$ be a ribbon complex, which is a CW complex. From Lemma~\ref{lemma:CWcomplex}, each map $f:\rbx E\longrightarrow \rbx E$ has a fixed point.
\end{proof}

The gradient (angle $\theta_p$ of the tangent) of a fixed point $p$ on a ribbon cycle boundary is a useful source of a distinguishing characteristic of a ribbon, defined by
\[
\overbrace{\theta_p = tan^{-1}\left[\frac{\frac{\partial f}{\partial y}}{\frac{\partial f}{\partial x}}\right].}^{\mbox{\textcolor{blue}{\bf Gradient angle for fixed point $p$ on inner ribbon boundary}}}
\]

\begin{example}
Let $\rb E$ be a ribbon complex with a fixed point $f(p) = p$ on a ribbon boundary with gradient angle $\theta_p$ in a CW space $K$ and let $\rb E'$ be a ribbon complex with a fixed point $g(q) = q$ on a ribbon boundary with gradient angle $\theta_q$ in a CW space $K'$. Let $\Phi(\rb E) = \theta_p$, $\Phi(\rb E') = \theta_q$ and threshold $th > 0$.  Then $\rb E\ \dxnear\ \rb E'$ if and only if $\norm{\Phi(\rb E) - \Phi(\rb E')} < th$, {\em i.e.}, the fixed points $p,q$ have close gradient angles.
\qquad \textcolor{blue}{\Squaresteel}  
\end{example}  

\subsection{Ribbon and Ribbon Nerve Betti numbers}$\mbox{}$\\
There are three basic types of Betti numbers that have intuitive meaning, namely, $\mathcal{B}_0$ (number of cells in a complex),  $\mathcal{B}_1$ (number cycles in a complex) and $\mathcal{B}_1$ (number holes in a complex) ~\cite[\S 4.3.2, p. 57]{Zomorodian2001BettiNumbers}.   In terms of ribbons and ribbon nerves in CW spaces, Betti numbers that enumerate fundamental shape structures are useful, namely,$\mbox{}$\\
\begin{compactenum}[1$^o$]
\item [{\bf Ribbon Betti number}]$\mbox{}$\\ 
Denoted by $\boldsymbol{\mathcal{B}_{\mbox{\tiny \rb}}}$, which is a count of the number of filaments (edges attached between ribbon cycles) + number of ribbon holes + 2 cycles. Let $\rb E$ be a planar ribbon, which is a pair of nesting, non-concentric filled cycles. 

\begin{example}
In Fig.~\ref{fig:ribbonFixedPoint}, the structure of ribbon $\rb E$ contains a pair of nesting cycles $\cyc A, \cyc B$, a single filament $\arc{pq}$ attached to the cycles and 3 holes (represented by \textcolor{gray}{\large $\boldsymbol{\bullet}$}).  Hence, $\mathcal{B}_{\mbox{\tiny \rb}}(\rb E) = \mathcal{B}_{0}(\rb E) + \mathcal{B}_{1}(\rb E) + \mathcal{B}_{2}(\rb E)= 1 + 2 + 3 = 6$.
\qquad \textcolor{blue}{\Squaresteel}
\end{example}

\item [{\bf Vergili Ribbon complex Betti number}]$\mbox{}$\\ 
Denoted by $\boldsymbol{\mathcal{B}_{\mbox{\tiny \rbx}}}$, which is a count of the number of ribbons in a Vergili ribbon complex. 

\begin{example}
In Fig.~\ref{fig:ribbonComplex}, the structure of Vergili ribbon complex $\rbx K$contains 5 ribbons.  Hence, $\mathcal{B}_{\mbox{\tiny \rbx}}(\rbx K) = 5$.
\qquad \textcolor{blue}{\Squaresteel}
\end{example} 

\item [{\bf Ribbon nerve Betti number}]$\mbox{}$\\ 
Denoted by $\boldsymbol{\mathcal{B}_{\mbox{\tiny \rbNrv}}}$, which is a count of the number of
filaments (edges attached between adjacent pairs of nerve cycles) + number of nerve holes + number of 
overlapping (intersecting) ribbons.

\begin{example}
In Fig.~\ref{fig:rbNrvE1}, the structure of a ribbon nerve $\rbNrv E$ contains 3 intersecting ribbons $\rb A, \rb B, \rb B'$, zero filaments and
6 holes (represented by \textcolor{gray}{\large $\boldsymbol{\bullet}$}). Hence, $\mathcal{B}_{\mbox{\tiny \rbNrv}} = 0 + 3 + 3 = 6$.
\qquad \textcolor{blue}{\Squaresteel}
\end{example}
\end{compactenum}

\begin{lemma}\label{lemma:BettiNumbers}
Let $\mathcal{B}_{0},\mathcal{B}_{1},\mathcal{B}_{2}$ be Betti numbers that count the number of cells, number of cycles and number of holes  in a planar CW complex, respectively.   Then
\begin{compactenum}[1$^o$]
\item $\mathcal{B}_{\mbox{\tiny \rb}}(\rb E) = \mathcal{B}_{0}(\rb E) + \mathcal{B}_{2}(\rb E)$ + 2 for a ribbon $\rb E$.
\item $\mathcal{B}_{\mbox{\tiny \rbx}}(\rbx K) = \mathop{\sum}\limits_{\rb E\in \rbx K} \mathcal{B}_{\mbox{\tiny \rb}}(\rb E)$ for a ribbon complex $\rbx K$ containing ribbons $\rb E$. 
\item $\mathcal{B}_{\mbox{\tiny \rbNrv}}(\rbNrv E) = \mathcal{B}_{0}(\rbNrv E) + \mathcal{B}_{1}(\rbNrv E) + \mathcal{B}_{2}(\rbNrv E)$ for a ribbon nerve $\rbNrv E$.
\end{compactenum}
\end{lemma}
\begin{proof}$\mbox{}$\\
1$^o$: By definition, a \emph{ribbon} $\rb E$ is a pair of nesting, non-concentric cycles $\cl(\cyc A)\supset \cl(\cyc B)\setminus \Int(\cl(\cyc B))$ with possible holes in the interior $\Int(\rb E)$ (counted with $\mathcal{B}_{2}(\rb E)$) and possible filaments (edges) attached between the ribbon cycles (counted with $\mathcal{B}_{0}(\rb E)$).   Consequently, $\mathcal{B}_{\mbox{\tiny \rb}}(\rb E)$ is a count of the number ribbon filaments (zeroth Betti number $\mathcal{B}_{0}(\rb E)$) + ribbon holes (twoth Betti number $\mathcal{B}_{2}(\rb E)) + 2$.\\
2$^o$: Immediate from Def.~\ref{def:ribbonComplex}.\\
3$^o$: By definition, a \emph{ribbon nerve} $\rbNrv E$ is a collection of cycles with a common part. The interior of each nerve ribbon may have holes and each nerve ribbon may have edges attached between ribbon cycle vertexes. Consequently, $\mathcal{B}_{\mbox{\tiny \rbNrv}}(\rb E)$ is a count of the number nerve ribbon filaments (zeroth Betti number $\mathcal{B}_{0}(\rb E)$) + nerve ribbon cycles (oneth Betti number $\mathcal{B}_{1}(\rb E)$) + nerve ribbon holes (twoth Betti number $\mathcal{B}_{2}(\rb E)$).  
\end{proof}

\begin{theorem}
Let $\vNrv E$ be a vortex nerve in a CW space.  Then 
\begin{compactenum}[1$^o$]
\item A Betti number that counts the number of ribbons in nerve $\vNrv E$ (denoted by $\mathcal{B}_{\mbox{\tiny \rb-\vNrv}}$) is defined by
\[
\mathcal{B}_{\mbox{\tiny \rb-\vNrv}}(\vNrv E) = \mathop{\sum}\limits_{\rb A\in \vNrv E}\mathcal{B}_{\mbox{\tiny \rb}}(\rb A).
\]
\item A Betti number that counts the number of ribbon nerves in nerve $\vNrv E$ (denoted by $\mathcal{B}_{\mbox{\tiny \rbNrv-\vNrv}}(\vNrv E)$) is defined by
\[
\mathcal{B}_{\mbox{\tiny \rbNrv-\vNrv}}(\vNrv E) = \mathop{\sum}\limits_{\rbNrv A\in \vNrv E}\mathcal{B}_{\mbox{\tiny \rbNrv}}(\rbNrv A).
\]
\end{compactenum}
\end{theorem}
\begin{proof}$\mbox{}$\\
1$^o$: Immediate from Theorem~\ref{thm:vortexNerveRibbons}, since each vortex nerve is collection of pairs of cycles such that the boundary of the inner ribbon cycle is a boundary of the interior of the outer cycle.\\
2$^o$: Immediate from Theorem~\ref{thm:vortexNerveRibbonNerves}, since each vortex nerve is collection of pairs of overlapping ribbons such that each pair of ribbons has a common cycle.
\end{proof}

\subsection{Homotopic Types of Ribbon Complexes and Ribbon Nerves}

The results in this section stem from the Edelsbrunner-Harer Theorem~\ref{EHnerve} for homotopy types.

\begin{theorem}\label{EHnerve}{\rm ~\cite[\S III.2, p. 59]{Edelsbrunner1999}}
Let $\mathscr{F}$ be a finite collection of closed, convex sets in Euclidean space.  Then the nerve of $\mathscr{F}$ and the union of the sets in $\mathscr{F}$ have the same homotopy type.
\end{theorem}

\begin{theorem}\label{thm:vortexNerveHomotopy}
Let $K$ be a collection of closed, convex Vergili ribbon complexes $\rbx K$ in Euclidean space.
Then each nerve of $\rbNrv K = \left\{\rbx K\in K: \bigcap \rbx K \neq \emptyset\right\}$ and the union of the Vergili ribbon complexes $\rbx K$ in $K$ have the same homotopy type.
\end{theorem}
\begin{proof}$\mbox{}$\\
From Theorem~\ref{EHnerve}, we have that the union of the Vergili ribbon complexes $\rbx K$ in $\rbNrv K$ and ribbon nerve $\rbNrv K$ have the same homotopy type.
\end{proof}

From Theorem~\ref{EHnerve}, we obtain a fundamental result for ribbon nerves.

\begin{theorem}\label{thm:dnearVortexNerveHomotopy}
Let $K$ be a collection of ribbon nerves $\rbNrv K$ that are closed, convex complexes in Euclidean space.
Then the nerve $\Nrv K = \left\{\rbNrv K\in K: \bigcap \rbNrv K \neq \emptyset\right\}$ of $K$ and the union of the ribbon nerves $\rbNrv K$ in nerve $\Nrv K$ have the same homotopy type.
\end{theorem}
\begin{proof}$\mbox{}$\\
From Theorem~\ref{EHnerve}, we have that the union of the ribbon nerves $\rbNrv K$ in $K$ and nerve $\Nrv K$ have the same homotopy type.
\end{proof}

\section*{Acknowledgements}
Many thanks to Tane Vergili, Sheela Ramanna and the anonymous reviewers for their incisive comments and helpful suggestions for this paper.

\bibliographystyle{amsplain}
\bibliography{NSrefs}

\end{document}